\documentclass[reqno]{amsart}
\usepackage{amssymb, url }
\usepackage{hyperref}

\theoremstyle{definition}
\newtheorem{definition}{Definition}[section]

\newtheorem{remark}{Remark}[section]
\theoremstyle{plain}
\newtheorem{theorem}[definition]{Theorem}
\newtheorem{lemma}[definition]{Lemma}
\newtheorem{proposition}[definition]{Proposition}
\newtheorem{corollary}[definition]{Corollary}

\allowdisplaybreaks[4]
\DeclareMathOperator{\T}{\mathbb T}
\DeclareMathOperator{\Z}{\mathbb Z}
\DeclareMathOperator{\sigt}{\sigma (t)}
\DeclareMathOperator{\rt}{\rho (t)}

\DeclareMathOperator{\fd}{{\it f}^\Delta (t)}

\DeclareMathOperator{\fs}{{\it f}^\sigma (t)}

\DeclareMathOperator{\Fd}{{\it F}^\Delta (t)}

\begin{document}

\title[A new generalization of Ostrowski type inequality on time scales]
{A new generalization of Ostrowski type inequality on time scales}

\author[W. J. Liu]{Wenjun Liu}
\address[W. J. Liu]{College of Mathematics and Physics\\
Nanjing University of Information Science and Technology \\
Nanjing 210044, China} \email{\href{mailto: W. J. Liu
<wjliu@nuist.edu.cn>}{wjliu@nuist.edu.cn}}

\author[Q. A. Ng\^{o}]{Qu\^{o}\hspace{-0.5ex}\llap{\raise 1ex\hbox{\'{}}}\hspace{0.5ex}c Anh Ng\^{o}}
\address[Q. A. Ng\^{o}]{Department of Mathematics, Mechanics and Informatics\\
College of Science\\ Vi\d{\^{e}}t Nam National University\\ H\`{a}
N\d{\^{o}}i, Vi\d{\^{e}}t Nam} \email{\href{mailto: Q. A. Ng\^{o}
<bookworm\_vn@yahoo.com>}{bookworm\_vn@yahoo.com}}

\author[W. B. Chen]{Wenbing Chen}
\address[W. B. Chen]{College of Mathematics and Physics\\
Nanjing University of Information Science and Technology \\
Nanjing 210044, China} \email{\href{mailto: W. B. Chen
<chenwb@nuist.edu.cn>}{chenwb@nuist.edu.cn}}

\subjclass[2000]{26D15; 39A10; 39A12; 39A13.}

\keywords{Ostrowski's inequality; generalization; time scales;
Simpson inequality; trapezoid inequality; mid-point inequality.}

\begin{abstract}
In this paper we first extend a generalization of Ostrowski type
inequality on time scales for functions whose derivatives are
bounded and then unify corresponding continuous and discrete
versions. We also point out some particular integral type
inequalities on time scales as special cases.

\end{abstract}

\thanks{This paper was typeset using \AmS-\LaTeX}

\maketitle

\section{introduction}

 In 1938, Ostrowski derived the following interesting integral inequality.

\begin{theorem}\label{t1}
Let $f: [a,b]\rightarrow \mathbb{R}$ be continuous on $[a,b]$ and differentiable in $(a,b)$ and its derivative $f': (a,b)\rightarrow \mathbb{R}$ is bounded in $(a,b)$, that is, $\|f'\|_\infty:=\sup\limits_{t\in(a,b)}|f'(x)|<
\infty$. Then for any $x\in [a,b]$, we have the inequality:
\begin{equation}\label{eq1}
\left|f(x)-\frac{1}{b-a}\int\limits_a^bf(t)dt\right|\leq \left(\frac{1}{4}+\frac{\big(x-\frac{a+b}{2}\big)^2}{(b-a)^2}\right)(b-a)\|f'\|_\infty.
\end{equation}
The inequality is sharp in the sense that the constant $\frac{1}{4}$ cannot be replaced by a smaller one.
\end{theorem}

For some extensions, generalizations and similar results,  see
\cite{dcr, l1, l3, l4, mpf, mpf2} and references therein.

The development of the theory of time scales was initiated by Hilger
\cite{h1988} in 1988 as a theory capable to contain both difference
and differential calculus in a consistent way. Since then, many
authors have studied the theory of certain integral inequalities on
time scales. For example, we refer the reader to \cite{abp2001,
bm2007, bm2008, OSY, wyy}. In \cite{bm2008}, Bohner and Matthews
established  the following so-called Ostrowski's inequality on time
scales.

\begin{theorem}[See \cite{bm2008}, Theorem 3.5]\label{th2}
Let $a, b, s, t\in \T$, $a<b$ and $f: [a, b]\rightarrow \mathbb{R}$ be differentiable. Then
\begin{equation}\label{eq2}
 \left|f(t)-\frac{1}{b-a}\int\limits_a^b f^\sigma(s)\Delta  s\right|   \leq \frac{M}{b-a}\Big(h_2(t,a)+h_2(t,b)\Big),
\end{equation}
where $M=\sup\limits_{a<t<b}|f^\Delta(t)|.$ This inequality is sharp
in the sense that the right-hand side of (\ref{eq2}) cannot be
replaced by a smaller one.
\end{theorem}

More recently, the authors proved the Ostrowski-Gr\"{u}ss type
inequality on time scales \cite{ln}.

\begin{theorem}\label{th3}
Let $a, b, s, t\in \T$, $a<b$ and $f: [a, b]\rightarrow \mathbb{R}$ be differentiable. If $f^\Delta$ is rd-continuous and
$$\gamma\leq f^\Delta(t)\leq \Gamma,\ \ \ \forall\ t\in [a, b].$$
Then we have
\begin{equation}\label{eq3}
 \left|f(t)-\frac{1}{b-a}\int\limits_a^b f^\sigma(s)\Delta  s- \frac{f(b)-f(a)}{(b-a)^2}\Big( h_2(t,a)- h_2 (t, b) \Big)\right|
   \leq \frac{1}{4}(b-a)(\Gamma-\gamma)
\end{equation}
for all $t\in [a, b]$.
\end{theorem}

In the present paper, by introducing a parameter, we first extend
a generalization of Ostrowski type inequality on time scales for
functions whose derivatives are bounded and then unify
corresponding continuous and discrete versions. We also point out
some particular integral type inequalities on time scales as
special cases.

\section{Time scales essentials}

Now we briefly introduce the time scales theory and refer the reader
to Hilger \cite{h1988} and the books \cite{bp2001, bp2003, LSK} for
further details.

\begin{definition}
{\it A time scale} $\T$ is an arbitrary nonempty closed subset of real numbers.
\end{definition}


\begin{definition}
 For $t \in \T$, we define the {\it forward jump operator} $\sigma : \T \to \T$ by
$
\sigt = \inf \left\{ {s \in \T:s > t} \right\},
$
while the {\it backward jump operator} $\rho : \T \to \T$ is defined by
$
\rt = \sup \left\{ {s \in \T:s < t} \right\}.
$
If $\sigt >t $, then we say that $t$ is {\it right-scattered}, while if $\rt < t$ then we say that $t$ is {\it left-scattered}.
\end{definition}

Points that are right-scattered and left-scattered at the same time
are called isolated. If $\sigt = t$, the $t$ is called {\it
right-dense}, and if $\rt = t$ then $t$ is called {\it left-dense}.
Points that are both right-dense and left-dense are called dense.

\begin{definition}
Let $t \in \T$, then two mappings $\mu ,\nu :\T \to \left[ {0, +
\infty } \right)$ satisfying $$ \mu \left( t \right): = \sigt - t,
 \, \nu \left( t \right): = t - \rt$$ are called the {\it
graininess functions}.
\end{definition}

We now introduce the set $\T^\kappa$ which is derived from the time
scales $\T$ as follows. If $\T$ has a left-scattered maximum $t$,
then $\T^\kappa := \T -\{t\}$, otherwise $\T^\kappa := \T$.
Furthermore for a function $f : \T \to \mathbb R$, we define the
function $f^\sigma : \T \to \mathbb R$ by $\fs = f(\sigma(t))$ for
all $t \in \T$.

\begin{definition}
Let $f : \T \to \mathbb R$ be a function on time scales. Then for $t \in \T^\kappa$, we define $\fd$ to be the number, if one exists, such that for all $\varepsilon >0$ there is a neighborhood $U$ of $t$ such that for all $s \in U$
\[
\left| {\fs - f\left( s \right) - \fd \left( {\sigt - s} \right)}
\right| \leq \varepsilon \left| {\sigt - s} \right|.
\]
We say that $f$ is $\Delta$-differentiable on $\T^\kappa$ provided
$\fd$ exists for all $t \in \T^\kappa$.
\end{definition}


\begin{definition} A mapping $f : \T \to \mathbb R$ is called {\it rd-continuous} (denoted by $C_{rd}$) provided if it satisfies
  \begin{enumerate}
   \item $f$ is continuous at each right-dense point or maximal element of $\T$.
   \item The left-sided limit $\mathop {\lim }\limits_{s \to t - } f\left( s \right) = f\left( {t - } \right)$ exists at each left-dense point $t$ of $\T$.
  \end{enumerate}
\end{definition}

\begin{remark}
It follows from Theorem 1.74 of Bohner and Peterson \cite{bp2001} that every rd-continuous function has an anti-derivative.
\end{remark}

\begin{definition} A function $F : \T \to \mathbb R$ is called a $\Delta$-antiderivative
of $f : \T \to \mathbb R$ provided $\Fd =f(t)$ holds for all $t \in \T^\kappa$.
Then the $\Delta$-integral of $f$ is defined by
\[
\int\limits_a^b {f\left( t \right)\Delta t} = F\left( b \right) - F\left( a \right).
\]
\end{definition}

\begin{proposition} \label{pro1}
Let $f, g$ be rd-continuous, $a, b, c\in \mathbb{T}$ and $\alpha, \beta\in \mathbb{R}$. Then
\begin{enumerate}
  \item $\int\limits_a^b {[\alpha f(t)+\beta g(t)]\Delta t} = \alpha\int\limits_a^b {f(t)\Delta t}+\beta\int\limits_a^b {g(t)\Delta t}, $
  \item $\int\limits_a^b {f(t)\Delta t}=-\int\limits_b^a {f(t)\Delta t},$
  \item $\int\limits_a^b {f(t)\Delta t}=\int\limits_a^c {f(t)\Delta t}+\int\limits_c^b {f(t)\Delta t},$
  \item $\int\limits_a^b {f(t)g^\Delta(t)\Delta t}=(fg)(b)-(fg)(a)-\int\limits_a^b {f^\Delta(t)g(\sigma(t))\Delta t},$
  \item $\int\limits_a^a {f(t)\Delta t}=0.$
 \end{enumerate}
\end{proposition}

\begin{definition}\label{de7}
Let $h_k : \T^2 \to \mathbb R$, $k \in \mathbb N_0$ be defined by
\[
h_0 \left( {t,s} \right) = 1 \quad {\text{ for all }} \quad s,t \in
\T
\]
and then recursively by
\[
h_{k + 1} \left( {t,s} \right) = \int\limits_s^t {h_k \left( {\tau
,s} \right)\Delta \tau } \quad {\text{ for all }} \quad s,t \in \T.
\]
\end{definition}

%

\section{ The  Ostrowski type inequality on time scales}

Our main result reads as follow.

\begin{theorem}\label{th4}
Let $a, b, s, t\in \T$, $a<b$ and $f: [a, b]\rightarrow \mathbb{R}$
be  differentiable. Then
\begin{equation} \label{eq4}
\begin{split}
\Bigg|
(1-\lambda)f(t)+\lambda\frac{f(a)+f(b)}{2}&-\frac{1}{b-a}\int\limits_a^b
f^\sigma(s)\Delta
 s \Bigg| \\
\leq \frac{M}{b-a} \Bigg( &h_2 \left( {a,a + \lambda \frac{{b - a}}
{2}} \right) + h_2 \left( {t,a + \lambda \frac{{b - a}}
{2}} \right) \\
+ &h_2 \left( {t,b - \lambda \frac{{b - a}}
{2}} \right) + h_2 \left( {b,b - \lambda \frac{{b - a}} {2}}
\right) \Bigg)
\end{split}
\end{equation}
for all $\lambda\in [0, 1]$ and $t\in [a+\lambda \frac{b-a}{2}, b-\lambda \frac{b-a}{2}]\cap \T$, where
\[
M:=\sup\limits_{a<t<b}|f^{\Delta}(t)|<\infty.
\]
This inequality is sharp provided
\begin{equation}\label{eq5}
\frac{\lambda }{2}a(b-a) + \frac{{\lambda ^2 }}{4}\left( {b - a}
\right)^2 \leqslant \int\limits_a^{a + \lambda \frac{{b - a}} {2}}
{s\Delta s}.
\end{equation}
\end{theorem}
\begin{remark}
We note that the condition \eqref{eq5} is trivial if $\lambda = 0$.
\end{remark}

To prove Theorem \ref{th4}, we need the following Generalized Montgomery Identity.

\begin{lemma}[Generalized Montgomery Identity]\label{le1}
Under the assumptions of Theorem \ref{th4}, we have
\[
(1-\lambda)f(t)+\lambda\frac{f(a)+f(b)}{2}=\frac{1}{b-a}\int\limits_a^b
f^\sigma(s)\Delta
 s +
\frac{1}{b-a}\int\limits_a^b K(t,s)f^{\Delta}(s)\Delta s,
\]
where
\begin{equation}\label{eq6}
 K(t,s)=\left\{ \begin{array}{ll}
 \displaystyle s-\left(a+\lambda\frac{b-a}{2}\right),\ \ &s\in [a, t), \hfill \medskip\\
  \displaystyle  s-\left(b-\lambda\frac{b-a}{2}\right),\ \ &s\in [t, b]. \hfill
\end{array} \right.
\end{equation}
\end{lemma}

\begin{proof}
Integrating by parts and applying Property \ref{pro1}, we have
\begin{align*}
 &\int\limits_a^b K(t,s)f^{\Delta}(s)\Delta
s\\
 = &\int\limits_a^t \left(s-\left(a+\lambda\frac{b-a}{2}\right)\right)f^{\Delta}(s)\Delta
s+\int\limits_t^b
\left(s-\left(b-\lambda\frac{b-a}{2}\right)\right)f^{\Delta}(s)\Delta
s\\
 = &
 \left(t-\left(a+\lambda\frac{b-a}{2}\right)\right)f(t)+\frac{\lambda}{2}(b-a)f(a)-\int\limits_a^t f^{\sigma}(s)\Delta
s\\&-\left(t-\left(b-\lambda\frac{b-a}{2}\right)\right)f(t)+\frac{\lambda}{2}(b-a)f(b)-\int\limits_t^b
f^{\sigma}(s)\Delta s
 \\
 = & (b-a)\left((1-\lambda)f(t)+\lambda\frac{f(a)+f(b)}{2}\right)-\int\limits_a^b f^{\sigma}(s)\Delta
s,
\end{align*}
from which we get the desired identity.
\end{proof}

\begin{corollary}[Continuous case]\label{co1}
Let $\T= \mathbb{R}$. Then
\begin{equation}\label{eq7}
(1-\lambda)f(t)+\lambda\frac{f(a)+f(b)}{2}=\frac{1}{b-a}\int\limits_a^b
f(s)\,{\rm d} s+ \frac{1}{b-a}\int\limits_a^b K(t,s)f'(s)\,{\rm d}
s.
\end{equation}
\end{corollary}

\begin{remark}
This is the Montgomery identity in the continuous case, which can be
found in \cite{dcr}.
\end{remark}

\begin{corollary}[Discrete case]\label{co2}
Let $\T= \Z$, $a=0$, $b=n$, $s=j$, $t=i$ and $f(k)=x_k$. Then
\[
(1-\lambda)x_{i} +\lambda\frac{x_0+x_n}{2} =
\frac{1}{n}\sum\limits_{j = 1}^n {x_j } + \frac{1}
{{n}}\sum\limits_{j = 0}^{n - 1} {K\left( {i,j} \right)\Delta  x_j }
,
\]
where
\begin{align*}
K(i, 0)&=-\frac{n\lambda}{2},\\
K(1, j)&=j-\left(n-\frac{n\lambda}{2}\right) \ \ \mbox{for}\ \ 1\leq j\leq n-1,\\
K(n, j)&=j- \frac{n\lambda}{2}  \ \ \mbox{for}\ \ 0\leq j\leq n-1,\\
K(i,j)&=\left\{ \begin{array}{cl}
 j- \frac{n\lambda}{2},\ \ &j\in [0, i),\medskip \hfill \\
   j-\left(n-\frac{n\lambda}{2}\right),\ \ &j\in [i, n-1], \hfill
\end{array} \right.
\end{align*}
as we just need $1\leq i\leq n$ and $1\leq j\leq n-1.$
\end{corollary}

\begin{corollary}[Quantum calculus case]\label{co3}
Let $\T= q^{\mathbb{N}_0}$, $q>1$, $a=q^m, b=q^n$ with $m<n$. Then
\begin{align*}
(1-\lambda)f(t)+&\lambda\frac{f(q^m)+f(q^n)}{2}=\frac{\sum\limits_{k=m}^{n-1}q^kf(q^{k+1})}{\sum\limits_{k=m}^{n-1}q^k}+\frac{1}{q^n-q^m}\sum\limits_{k=m}^{n-1}\left[f(q^{k+1})-f(q^{k})\right]K(t,q^k),
\end{align*}
where
\[
 K(t,q^k)=\left\{ \begin{array}{cl}
 \displaystyle q^k-\left(q^m+\lambda\frac{q^m-q^n}{2}\right),\ \ &q^k\in [q^m, t), \hfill \medskip\\
  \displaystyle q^k-\left(q^n-\lambda\frac{q^m-q^n}{2}\right),\ \ &q^k\in [t, q^n]. \hfill
\end{array} \right.
\]
\end{corollary}

\begin{proof}[Proof of Theorem~\ref{th4}]
By applying Lemma \ref{le1}, we get
\begin{align*}
 &\Bigg|(1-\lambda)f(t)+\lambda\frac{f(a)+f(b)}{2}- \frac{1}{b-a} \int\limits_a^b f^{\sigma}(s)\Delta
s\Bigg|\\
\leq &\frac{1}{b-a}\left|\int\limits_a^b K(t,s)f^{\Delta}(s)\Delta s\right|\\
  \leq &\frac{M}{b-a} \left(\int\limits_a^t {\left| {K\left( {t,s} \right)} \right|\Delta s}  + \int\limits_t^b {\left| {K\left( {t,s} \right)} \right|\Delta
  s}\right)
    \hfill \\
   = &\frac{M}{b-a} \left(\int\limits_a^t {\left| {s - \left( {a + \lambda \frac{{b - a}}
{2}} \right)} \right|\Delta s}  + \int\limits_t^b {\left| {s - \left( {b -
\lambda \frac{{b - a}}
{2}} \right)} \right|\Delta s} \right) \hfill \\
   = &\frac{M}{b-a} \left(\int\limits_a^{a + \lambda \frac{{b - a}}
{2}} {\left| {s - \left( {a + \lambda \frac{{b - a}} {2}} \right)}
\right|\Delta s}  + \int\limits_{a + \lambda \frac{{b - a}} {2}}^t {\left|
{s - \left( {a + \lambda \frac{{b - a}}
{2}} \right)} \right|\Delta s}\right.  \hfill \\
   &\left.+ \int\limits_t^{b - \lambda \frac{{b - a}}
{2}} {\left| {s - \left( {b - \lambda \frac{{b - a}} {2}} \right)}
\right|\Delta s}  + \int\limits_{b - \lambda \frac{{b - a}} {2}}^b {\left|
{s - \left( {b - \lambda \frac{{b - a}}
{2}} \right)} \right|\Delta s}\right)  \hfill \\
   = &\frac{M}{b-a} \left(\int\limits_{a + \lambda \frac{{b - a}}
{2}}^a {\left( {s - \left( {a + \lambda \frac{{b - a}} {2}} \right)}
\right)\Delta s}  + \int\limits_{a + \lambda \frac{{b - a}} {2}}^t {\left(
{s - \left( {a + \lambda \frac{{b - a}}
{2}} \right)} \right)\Delta s}\right.  \hfill \\
  &\left. + \int\limits_{b - \lambda \frac{{b - a}}
{2}}^t {\left( {s - \left( {b - \lambda \frac{{b - a}} {2}} \right)}
\right)\Delta s}  + \int\limits_{b - \lambda \frac{{b - a}} {2}}^b {\left(
{s - \left( {b - \lambda \frac{{b - a}}
{2}} \right)} \right)\Delta s} \right) \hfill \\
   = &\frac{M}{b-a} \left(h_2 \left( {a,a + \lambda \frac{{b - a}}
{2}} \right) + h_2 \left( {t,a + \lambda \frac{{b - a}}
{2}} \right)\right. \hfill \\
   &\left.+ h_2 \left( {t,b - \lambda \frac{{b - a}}
{2}} \right) + h_2 \left( {b,b - \lambda \frac{{b - a}} {2}}
\right)\right),
\end{align*}
which completes the first part of our proof. To prove the sharpness of this inequality, let $f(t)=t$,
$t=b-\lambda \frac{b-a}{2}$. It follows that $M=1$. Starting with
the righ-hand side of \eqref{eq4}, we have
\begin{align*}
\frac{M}{b-a} \Bigg(&h_2 \left( {a,a + \lambda \frac{{b - a}}
{2}} \right) + h_2 \left( {t,a + \lambda \frac{{b - a}}
{2}} \right)\\
   & +h_2 \left( {t,b - \lambda \frac{{b - a}}
{2}} \right) + h_2 \left( {b,b - \lambda \frac{{b - a}} {2}}
\right)\Bigg)\\
=  \frac{1}{b-a} \Bigg(&h_2 \left( {a,a + \lambda \frac{{b - a}}
{2}} \right) + h_2 \left( {b - \lambda \frac{{b - a}} {2},a + \lambda \frac{{b - a}}
{2}} \right) + h_2 \left( {b,b - \lambda \frac{{b - a}} {2}}\right)\Bigg) .\\
\end{align*}
Moreover,
\begin{align*}
  h_2 \left( {a,a + \lambda \frac{{b - a}}
{2}} \right) &= \int\limits_{a + \lambda \frac{{b - a}}
{2}}^a {\left( {s - \left( {a + \lambda \frac{{b - a}}
{2}} \right)} \right)\Delta s}  \hfill \\
   &= \int\limits_{a +  \lambda \frac{{b - a}}
{2}}^a {s\Delta s}  - \left( {a + \lambda \frac{{b - a}}
{2}} \right)\left( {a - \left( {a + \lambda \frac{{b - a}}
{2}} \right)} \right) \hfill \\
   &= \int\limits_{a + \lambda \frac{{b - a}}
{2}}^a {s\Delta s}  + \left( {a + \lambda \frac{{b - a}}
{2}} \right)\lambda \frac{{b - a}}
{2}.
\end{align*}
\begin{align*}
  h_2 \left( {b - \lambda \frac{{b - a}}
{2},a + \lambda \frac{{b - a}}
{2}} \right) &= \int\limits_{a + \lambda \frac{{b - a}}
{2}}^{b - \lambda \frac{{b - a}}
{2}} {\left( {s - \left( {a + \lambda \frac{{b - a}}
{2}} \right)} \right)\Delta s}  \hfill \\
   &= \int\limits_{a + \lambda \frac{{b - a}}
{2}}^{b - \lambda \frac{{b - a}}
{2}} {s\Delta s}  - \left( {a + \lambda \frac{{b - a}}
{2}} \right)\left( {b - \lambda \frac{{b - a}}
{2} - \left( {a + \lambda \frac{{b - a}}
{2}} \right)} \right) \hfill \\
   &= \int\limits_{a +  \lambda \frac{{b - a}}
{2}}^{b - \lambda \frac{{b - a}}
{2}} {s\Delta s}  - \left( {a + \lambda \frac{{b - a}}
{2}} \right)\left( {b - a} \right)\left( 1- \lambda \right).
\end{align*}
\begin{align*}
  h_2 \left( {b,b - \lambda \frac{{b - a}}
{2}} \right) &= \int\limits_{b - \lambda \frac{{b - a}}
{2}}^b {\left( {s - \left( {b - \lambda \frac{{b - a}}
{2}} \right)} \right)\Delta s}  \hfill \\
   &= \int\limits_{b - \lambda \frac{{b - a}}
{2}}^b {s\Delta s}  - \left( {b - \lambda \frac{{b - a}}
{2}} \right)\left( {b - \left( {b - \lambda \frac{{b - a}}
{2}} \right)} \right) \hfill \\
   &= \int\limits_{b - \lambda \frac{{b - a}}
{2}}^b {s\Delta s} - \left( {b - \lambda \frac{{b - a}}
{2}} \right)\lambda \frac{{b - a}}
{2}.
\end{align*}
Thus, in this situation, the right-hand side of \eqref{eq4} equals to
\begin{align*}
 &  \frac{1}{b-a}\left(  - \int\limits_a^{a + \lambda \frac{{b - a}}
{2}} {s\Delta s}  + \int\limits_{a + \lambda \frac{{b - a}} {2}}^{b
- \lambda \frac{{b - a}} {2}} {s\Delta s}  + \int\limits_{b -
\lambda \frac{{b - a}}
{2}}^b {s\Delta s} \right)  \hfill \\
 &  + \left( {a + \lambda \frac{{b - a}}
{2}} \right) \frac{\lambda} {2} - \left( {a + \lambda \frac{{b - a}}
{2}} \right) \left( {1 - \lambda } \right) - \left( {b - \lambda
\frac{{b - a}} {2}} \right) \frac{\lambda}
{2} \hfill \\
     = & \frac{1}{b-a}\left( - 2\int\limits_a^{a + \lambda \frac{{b - a}}
{2}} {s\Delta s}  + \int\limits_{a }^{b}
{s\Delta s} \right)  - \left( {a + \lambda (b - a)} \right)\left( {1 - \lambda } \right).\hfill \\
\end{align*}
Starting with the left-hand side of \eqref{eq4}, we have
\begin{align*}
  \Bigg| \left( {1 - \lambda } \right)f\left( t \right) &+ \lambda \frac{{f\left( a \right) + f\left( b \right)}}
{2} - \frac{1}{b-a} \int\limits_a^b {f^\sigma  \left( s \right)\Delta s}  \Bigg| \\
   &= \left| {\left( {1 - \lambda } \right)\left( {b - \lambda \frac{{b - a}}
{2}} \right) + \lambda \frac{{a + b}}
{2} -  \frac{1}{b-a}  \int\limits_a^b {\sigma \left( s \right)\Delta s} } \right| \hfill \\
   &= \left| {\left( {1 - \lambda } \right)\left( {b - \lambda \frac{{b - a}}
{2}} \right) + \lambda \frac{{a + b}}
{2} +  \frac{1}{b-a}  \int\limits_a^b {s\Delta s}  - b - a} \right| \hfill \\
   &= \left| { - \lambda \left( {1 - \frac{\lambda }{2}} \right)\left( {b - a} \right) - a
   +  \frac{1}{b-a}  \int\limits_a^b {s\Delta s} } \right|,
\end{align*}
where  we have used
$$\int\limits_a^b {\sigma \left( s \right)\Delta
s}=\int\limits_a^b \big(\sigma ( s)+s\big)\Delta s-\int\limits_a^b s\Delta
s=\int\limits_a^b (s^2)^\Delta\Delta s-\int\limits_a^b s\Delta
s=b^2-a^2-\int\limits_a^b s\Delta s.$$
So, if
\[
\frac{\lambda }{2}a(b-a) + \frac{{\lambda ^2 }}{4}\left( {b - a}
\right)^2 \leqslant \int\limits_a^{a + \lambda \frac{{b - a}} {2}}
{s\Delta s}
\]
holds true, then
\begin{align*}
  \Bigg|  - \lambda \left( {1 - \frac{\lambda }
{2}} \right)&\left( {b - a} \right) - a + \frac{1}
{{b - a}}\int\limits_a^b {s\Delta s}  \Bigg|  \\
  & \geqslant  - \lambda \left( {1 - \frac{\lambda }
{2}} \right)\left( {b - a} \right) - a + \frac{1}
{{b - a}}\int\limits_a^b {s\Delta s} \\
 &  \geqslant \frac{1}{b-a}\left( - 2\int\limits_a^{a + \lambda \frac{{b - a}}
{2}} {s\Delta s}  + \int\limits_{a }^{b} {s\Delta s} \right)  -
\big( {a + \lambda (b - a)} \big)\left( {1 - \lambda } \right),
\end{align*}
which helps us to complete our proof.
 \end{proof}

If we apply the the inequality \eqref{eq4} to different time scales, we will get some well-known and some new results.

\begin{corollary}[Continuous case]\label{co4}
Let $\T= \mathbb{R}$. Then our delta integral is the usual
Riemann integral from calculus. Hence,
\[
h_2 \left( {t,s} \right) = \frac{{\left( {t - s} \right)^2 }}{2},
\quad {\text{ for all }} \quad t, s \in \mathbb R.
\]
This leads us to state the following inequality
\[
\begin{split}
\Bigg|(1-\lambda)f(t)+\lambda\frac{f(a)+f(b)}{2}&-\frac{1}{b-a}\int\limits_a^b
f(s)\,{\rm d} s\Bigg| \\
\leq &
M\left(\frac{1}{4}(b-a)\big( (1-\lambda)^2+\lambda^2 \big)+\frac{1}{b-a}\left(x-\frac{a+b}{2}\right)^2\right)
\end{split}
\]
for all $\lambda\in [0, 1]$ and $a+\lambda \frac{b-a}{2} \leq t\leq b-\lambda \frac{b-a}{2}$, where $M=\sup\limits_{x\in (a,b)} |f'(x)|<\infty$, which is exactly the generalized Ostrowski type inequality shown in Theorem 2  of \cite{dcr}.
\end{corollary}

\begin{corollary}[Discrete case]\label{co5}
Let $\T= \Z$, $a=0$, $b=n$, $s=j$, $t=i$ and $f(k)=x_k$. Thus, we have
\[
\left|(1-\lambda)x_{i} +\lambda\frac{x_0+x_n}{2}-
\frac{1}{n}\sum\limits_{j = 1}^n {x_j } \right| \leq
\frac{M}{n}\left(\left|i-\frac{n+1}{2}\right|^2+\frac{(2\lambda^2-2\lambda+1)n^2-1}{4}\right)
\]
for all $i \in \left[\frac{\lambda n}{2}, n- \frac{\lambda n}{2} \right] \cap \T$, where $M=\max\limits_{1 \leq i \leq n-1}
|\Delta x_i|<\infty$.
\end{corollary}
\begin{proof}
In this situation,  it is known that
\[
h_k \left( {t,s} \right) = \left( {\begin{array}{*{20}c}
  {t - s} \medskip\\
  k \\
 \end{array} } \right) , \quad {\text{ for all }} \quad t,s \in \mathbb Z.
\]
Therefore,
\begin{align*}
h_2 \left( {a,a + \lambda \frac{{b - a}} {2}} \right)&= \left(
{\begin{array}{*{20}c}
   -\frac{n\lambda}{2}\medskip \\
  2 \\
 \end{array} } \right) = \frac{\frac{n\lambda}{2} \left(\frac{n\lambda}{2}+1 \right)}
{2},\\
h_2 \left( {t,a + \lambda \frac{{b - a}} {2}} \right)&= \left(
{\begin{array}{*{20}c}
  i-\frac{n\lambda}{2} \medskip\\
  2 \\
 \end{array} } \right) = \frac{\left(i-\frac{n\lambda}{2} \right)\left(i-\frac{n\lambda}{2}-1 \right)}
{2},\\
h_2 \left( {t,b- \lambda \frac{{b - a}} {2}} \right)&= \left(
{\begin{array}{*{20}c}
  i-n+\frac{n\lambda}{2} \medskip\\
  2 \\
 \end{array} } \right) = \frac{\left(i-n+\frac{n\lambda}{2} \right)\left(i-n+\frac{n\lambda}{2}-1 \right)}
{2},
\end{align*}
and
\[
h_2 \left( {t,b- \lambda \frac{{b - a}} {2}} \right)= \left(
{\begin{array}{*{20}c}
  \frac{n\lambda}{2} \medskip\\
  2 \\
 \end{array} } \right) = \frac{\frac{n\lambda}{2}\left(\frac{n\lambda}{2}-1 \right)}
{2},
\]
Thus, we get the desired result. \end{proof}

\begin{corollary}\label{co6} {\rm (Quantum calculus case)}. Let
$\T= q^{\mathbb{N}_0}$, $q>1$, $a=q^m, b=q^n$ with $m<n$.
Then
\[
\begin{split}
 & \Bigg|(1-\lambda)f(t)+\lambda\frac{f(q^m)+f(q^n)}{2}-\frac{1}{q^n-q^m}\int\limits_{q^m}^{q^n}
f^\sigma(s)\Delta
 s\Bigg| \\
\leq
&\frac{M}{(1+q)(q^n-q^m)}\Bigg(2t^2-(1+q)(q^m+q^n)t \\
&+ \left(\left(2\lambda^2-\frac{3}{2}\lambda+1\right)(q^{2m+1}+q^{2n+1})
-\lambda(3-2\lambda)q^{m+n+1}+\frac{\lambda}{2}(q^m-q^n)^2\right) \Bigg)
\end{split}
\]
for all
\[
t \in \left[ {q^m  + \lambda \frac{{q^n  - q^m }}{2},q^n  - \lambda \frac{{q^n  - q^m }}{2}} \right] \cap \T
\]
where
$$M=\sup\limits_{t\in (q^m, q^n)}\left|\frac{f(qt)-f(t)}{(q-1)t}\right|.$$
\end{corollary}
\begin{proof}
In this situation, one has
\[
h_k \left( {t,s} \right) = \prod\limits_{\nu = 0}^{k - 1} {\frac{{t
- q^\nu s}} {{\sum\limits_{\mu = 0}^\nu {q^\mu } }}}, \quad {\text{
for all }} \quad t,s \in \T
\]
and
$$f^\Delta(t)=\frac{f(qt)-f(t)}{(q-1)t}.$$
Therefore,
\begin{align*}
h_2 \left( {q^m, q^m+\lambda\frac{q^n-q^m}{2} } \right) &=
\frac{\frac{\lambda}{2}{\left( {q^m - q^n } \right)\left[ q^m -
\left(1-\frac{\lambda}{2}\right)q^{m + 1} -\frac{\lambda}{2}q^{n +
1} \right]}} {{1 + q}},\\
h_2 \left( {t, q^m+\lambda\frac{q^n-q^m}{2} } \right) &=
\frac{\left[ t - \left(1-\frac{\lambda}{2}\right)q^{m}
-\frac{\lambda}{2}q^{n} \right]\left[ t -
\left(1-\frac{\lambda}{2}\right)q^{m + 1} -\frac{\lambda}{2}q^{n +
1} \right]} {{1 + q}},\\
h_2 \left( {t, q^n-\lambda\frac{q^n-q^m}{2} } \right) &=
\frac{\left[ t - \left(1-\frac{\lambda}{2}\right)q^{n}
-\frac{\lambda}{2}q^{m} \right]\left[ t -
\left(1-\frac{\lambda}{2}\right)q^{n + 1} -\frac{\lambda}{2}q^{m +
1} \right]} {{1 + q}},
\end{align*}
and
\[
h_2 \left( {q^m, q^n-\lambda\frac{q^n-q^m}{2} } \right) =
\frac{\frac{\lambda}{2}{\left( {q^n - q^m } \right)\left[ q^n -
\left(1-\frac{\lambda}{2}\right)q^{n + 1} -\frac{\lambda}{2}q^{m +
1} \right]}} {{1 + q}}.
\] Thus, we get the result. \end{proof}

\section{Some particular Ostrowski type inequalities on time scales}

In this section we point out some particular Ostrowski type
inequalities on time scales as special cases, such as: {\it
rectangle inequality} on time scales, {\it trapezoid inequality}
on time scales, {\it mid-point inequality} on time scales, {\it
Simpson inequality} on time scales,
 {\it averaged mid-point-trapezoid inequality} on time scales and
others.

Throughout this section, we always assume $\T$ is a time scale;
$a, b \in \T$ with $a < b$; $f: [a, b]\rightarrow \mathbb{R}$ is
differentiable.
 We denote $$M=\sup_{a<x<b}|f^\Delta(x)|.$$

\begin{corollary}\label{co7}
Under the assumptions of Theorem \ref{th4} with $\lambda=1$ and
$t=\frac{a+b}{2}\in \mathbb{T}$. Then we have the trapezoid
inequality on time scales
\begin{equation}\label{eq8}
 \left|\frac{f(a)+f(b)}{2}-\frac{1}{b-a}\int\limits_a^b f^\sigma(s)\Delta
 s\right|
  \leq \frac{M}{b-a}\left(h_2\left(a, \frac{a+b}{2}\right)+h_2\left(b, \frac{a+b}{2}\right)\right).
\end{equation}
\end{corollary}

\begin{remark}
 If we take $\lambda=0$ in Theorem \ref{th4}, then Theorem
 \ref{th2} is recaptured. Therefore, Theorem \ref{th4} may be
 regarded as a generalization of Theorem \ref{th2}.
\end{remark}

\begin{corollary}\label{co8}
Under the assumptions of Theorem \ref{th4} with
$\lambda=\frac{1}{3}$. Then we have the following integral
inequality on time scales
\begin{equation}\label{eq9}
\begin{split}
  \Bigg|\frac{1}{6}\left(f(a)+f(b)+4f(t)\right)-\frac{1}{b-a}&\int\limits_a^b f^\sigma(s)\Delta
 s\Bigg| \\
  \leq \frac{M}{b-a}\Bigg(&h_2\left(a, \frac{5a+b}{6}\right)+h_2\left(t, \frac{5a+b}{6}\right)\\
  +&h_2\left(t, \frac{a+5b}{6}\right)+h_2\left(b, \frac{a+5b}{6}\right)\Bigg)
\end{split}
\end{equation}
for all $t\in \left[\frac{5a+b}{6}, \frac{a+5b}{6}\right]\cap
\mathbb{T}$.
\end{corollary}

\begin{remark}
If we choose $t=\frac{a+b}{2}$ in (\ref{eq9}), we get the Simpson
inequality on time scales
\begin{equation*}
\begin{split}
 \Bigg|\frac{1}{6}\left(f(a)+4f\left(\frac{a+b}{2}\right)+f(b)\right)-\frac{1}{b-a}&\int\limits_a^b f^\sigma(s)\Delta
 s\Bigg| \\
  \leq \frac{M}{b-a}\Bigg(&h_2\left(a, \frac{5a+b}{6}\right)+h_2\left(\frac{a+b}{2}, \frac{5a+b}{6}\right)\\
  +&h_2\left(\frac{a+b}{2}, \frac{a+5b}{6}\right)+h_2\left(b,   \frac{a+5b}{6}\right)\Bigg).
\end{split}
\end{equation*}
\end{remark}

\begin{corollary}\label{co9}
Under the assumptions of Theorem \ref{th4} with
$\lambda=\frac{1}{2}$. Then we have the following integral
inequality on time scales
\begin{equation}\label{eq10}
\begin{split}
 \Bigg|\frac{1}{2}\left(\frac{f(a)+f(b)}{2}+f(t)\right)-\frac{1}{b-a}&\int\limits_a^b f^\sigma(s)\Delta
 s\Bigg| \\
  \leq \frac{M}{b-a}\Bigg(&h_2\left(a, \frac{3a+b}{4}\right)+h_2\left(t, \frac{3a+b}{4}\right)\\
  +&h_2\left(t, \frac{a+3b}{4}\right)+h_2\left(b,
  \frac{a+3b}{4}\right)\Bigg)
\end{split}
\end{equation}
for all $t\in \left[\frac{3a+b}{4}, \frac{a+3b}{4}\right]\cap
\mathbb{T}$.
\end{corollary}

\begin{remark}
If we choose $t=\frac{a+b}{2}$ in (\ref{eq10}), we get the
averaged mid-point-trapezoid inequality on time scales
\begin{equation*}
\begin{split}
 \Bigg| \frac{1}{2}\left(\frac{f(a)+f(b)}{2}+f\left(\frac{a+b}{2}\right)\right)&-\frac{1}{b-a}
 \int\limits_a^b f^\sigma(s)\Delta  s \Bigg| \\
  \leq \frac{M}{b-a}\Bigg(&h_2\left(a, \frac{3a+b}{4}\right)+h_2\left(\frac{a+b}{2}, \frac{3a+b}{4}\right)
  \\
+&h_2\left(\frac{a+b}{2}, \frac{a+3b}{4}\right)+h_2\left(b,  \frac{a+3b}{4}\right)\Bigg).
\end{split}
\end{equation*}
\end{remark}

\begin{corollary}\label{pro10}
Under the assumptions of Theorem \ref{th4} with
$t=\frac{a+b}{2}\in \mathbb{T}$. Then we have the following
integral inequality on time scales
\begin{equation}\label{eq11}
\begin{split}
\Bigg|(1-\lambda)f\left(\frac{a+b}{2}\right)&+\lambda\frac{f(a)+f(b)}{2}-\frac{1}{b-a}\int\limits_a^b
f^\sigma(s)\Delta  s\Bigg| \\
\leq \frac{M}{b-a} \Bigg(&h_2 \left( {a, a + \lambda \frac{{b -
a}} {2}} \right) + h_2 \left( {\frac{a+b}{2}, a + \lambda \frac{{b
- a}}
{2}} \right) \\
   + &h_2 \left( {\frac{a+b}{2}, b - \lambda \frac{{b - a}}
{2}} \right) + h_2 \left( {b, b - \lambda \frac{{b - a}} {2}}
\right)\Bigg)
\end{split}
\end{equation}
for all $\lambda\in [0, 1]$.
\end{corollary}

\begin{remark}
 If we choose $\lambda=0$ in (\ref{eq11}), we get the mid-point inequality on time scales
\[
 \left|f\left(\frac{a+b}{2}\right)-\frac{1}{b-a}\int\limits_a^b f^\sigma(t)\Delta  t\right|
   \leq \frac{M}{b-a}\Bigg(h_2\left(\frac{a+b}{2}, a\right)+h_2\left(\frac{a+b}{2}, b\right)\Bigg).
\]
\end{remark}

\medskip

\section*{Acknowledgements}

This work was supported by the Science Research Foundation of
Nanjing University of Information Science and Technology and the
Natural Science Foundation of Jiangsu Province Education Department
under Grant No.07KJD510133.


\begin{thebibliography}{99}

\bibitem{abp2001}
R. Agarwal, M. Bohner and A. Peterson, Inequalities on time scales: A survey, \textit{Math. Inequal. Appl.\/},~{\bf 4}(4) (2001), 535-557.

\bibitem{bp2001}
M. Bohner and A. Peterson, \textit{Dynamic Equations on Time Series}, Birkh\"auser, Boston, 2001.

\bibitem{bp2003}
M. Bohner and A. Peterson, \textit{Advances in Dynamic Equations on Time Series}, Birkh\"auser, Boston, 2003.

\bibitem{bm2007} M. Bohner and T. Matthewa, The
Gr\"{u}ss inequality on time scales,
 {\em Communications in Mathematical Analysis\/},~{\bf 3} (1) (2007), 1-8.

 \bibitem{bm2008} M. Bohner and T. Matthewa, Ostrowski inequalities on
time scales,
 {\em J. Inequal. Pure Appl. Math.\/},~{\bf 9} (1) (2008), Art. 6, 8 pp.

\bibitem{dcr}
S. S. Dragomir, P. Cerone and J. Roumelitis, A new generalization of
Ostrowski integral inequality for mappings whose derivatives are
bounded and applications in numerical integration and for special
means,   {\em Appl. Math. Lett.\/},~{\bf 13} (1) (2000), 19-25.




 \bibitem{h1988}
S. Hilger, \textit{Ein Ma$\beta$kettenkalk\"ul mit Anwendung auf
Zentrmsmannigfaltingkeiten}, PhD thesis, Univarsi. W\"urzburg, 1988.



\bibitem{LSK}
V. Lakshmikantham, S. Sivasundaram, and B. Kaymakcalan,
\textit{Dynamic Systems on Measure Chains}, Kluwer Academic Publishers, 1996.

\bibitem{l1}
W. J. Liu, Q. L. Xue and S. F. Wang,  Several new perturbed
Ostrowski-like type inequalities, {\em J. Inequal. Pure Appl.
Math.\/},~{\bf 8}(4) (2007), Art.110, 6 pp.

\bibitem{l3}
W. J. Liu,  C. C. Li and Y. M. Hao, Further generalization of some
double integral inequalities and applications, {\em Acta. Math.
Univ. Comenianae\/},~{\bf 77} (1)(2008), 147-154.


\bibitem{l4}
 W. J. Liu, Several error inequalities for a quadrature formula with a parameter and applications,
{\em Comput. Math. Appl.}, accepted.

\bibitem{ln}
 W. J. Liu and Q. A. Ng\^{o}, An Ostrowski-Gr\"{u}ss type inequality on time scales, arXiv: 0804.3231v1.

\bibitem{mpf}
D. S. Mitrinovi$\mathrm{\acute{c}}$, J. Pecari$\mathrm{\acute{c}}$
and A. M. Fink, {\it Inequalities for Functions and Their Integrals
and Derivatives}, Kluwer Academic, Dordrecht, (1994).

\bibitem{mpf2}
D. S. Mitrinovi$\mathrm{\acute{c}}$, J. Pecari$\mathrm{\acute{c}}$
and A. M. Fink, {\it Classical and New Inequalities in Analysis},
Kluwer Academic, Dordrecht, (1993).


\bibitem{OSY}
H. Roman, A time scales version of a Wirtinger-type inequality and
applications, Dynamic equations on time scales, {\it J. Comput.
Appl. Math.}, {\bf 141} (1/2) (2002), 219-226.

\bibitem{wyy}
F.-H. Wong, S.-L. Yu, C.-C. Yeh, Anderson¡¯s inequality on time
scales, \textit{Applied Mathematics Letters}, {\bf 19} (2007),
931-935.
\end{thebibliography}
\end{document}